\documentclass[letterpaper, 10 pt, conference]{ieeeconf}  

\IEEEoverridecommandlockouts                              
\usepackage{bbm}
\usepackage{mathrsfs}
\usepackage{verbatim}
\usepackage{amsmath}
\usepackage{amsfonts}
\usepackage{bm}
\usepackage{graphicx}
\usepackage{float}

\usepackage{amsthm}
\usepackage{xcolor}
\usepackage[colorinlistoftodos]{todonotes}
\usepackage{mathtools}
\usepackage{cite}
\usepackage{colortbl}
\usepackage{subfigure}

\theoremstyle{definition}
\newtheorem{assumption}{Assumption}
\newtheorem{definition}{Definition}
\newtheorem{remark}{Remark}
\newtheorem{problem}{Problem}

\theoremstyle{plain}

\newtheorem{theorem}{Theorem}
\newtheorem{lemma}{Lemma}

\title{\LARGE \bf Approximate Information States for Worst-case Control of Uncertain Systems}

\author{Aditya Dave, {\itshape{Student Member, IEEE,}} Nishanth Venkatesh, {\itshape{Student Member, IEEE,}}  \\
Andreas A. Malikopoulos, {\itshape{Senior Member, IEEE}} 
	\thanks{This research was supported by the Sociotechnical Systems Center (SSC) at the University of Delaware.} %
	\thanks{The authors are with the Department of Mechanical Engineering, University of Delaware, Newark, DE 19716 USA (email: \texttt{adidave@udel.edu; nish@udel.edu; andreas@udel.edu).}} }

\begin{document}

\maketitle
\thispagestyle{empty}

\begin{abstract}
In this paper, we investigate a worst-case-scenario control problem with a partially observed state. We consider a non-stochastic formulation, where noises and disturbances in our dynamics are uncertain variables which take values in finite sets. In such problems, the optimal control strategy can be derived using a dynamic program (DP) with respect to the memory. The computational complexity of this DP can be improved using a conditional range of the state instead of the memory. We present a more general definition of an information state which is sufficient to construct a DP without loss of optimality, and show that the conditional range is an example of an information state. Next, we extend this notion to define an approximate information state and an approximate DP. We also bound the maximum loss of optimality when using an approximate DP to derive the control strategy. Finally, we illustrate our results in a numerical example.
\end{abstract}

\section{Introduction}
\label{section:Introduction}

A decision maker in a cyber-physical system must generate control actions using noisy observations and uncertain predictions of the state's evolution \cite{kim2012cyber,Dave2020}. 
Such a decision making problem is usually modelled using one of the following approaches:
(1) Stochastic control: We consider that all uncertainties have known distributions and seek a control strategy which minimizes the expected total cost over a time horizon \cite{varaiya_book}. The resulting control strategy performs optimally on average across multiple runs of the system.
(2) Non-stochastic control: We consider that all uncertainties take values in known sets with unknown distributions and seek a control strategy which minimizes the maximum cost that can be incurred across a time horizon \cite{bernhard2003minimax}. The resulting control strategy is more conservative, but has concrete performance guarantees for any evolution of the system. Thus, this approach is more useful in applications where safety guarantees are critical, or where the probability distributions of external disturbances are unknown a priori.
Under either approach, the optimal control strategy can be derived offline using a dynamic program (DP) \cite{bernhard2003minimax}.
In general, the optimal action at any time is a function of the historical data in the decision maker's memory. This memory grows with time as more data is added. Subsequently, the domain of the corresponding control strategy grows with time which makes it computationally challenging to solve the DP for long time horizons. In stochastic control, this problem is alleviated by writing the DP using \textit{information states} instead of the memory \cite{subramanian2019approximate}. 
In fact, the notion of an information state is fundamental in the study of stochastic systems has also been extended to reinforcement learning \cite{littman2001predictive, subramanian2022approximate} and decentralized problems \cite{17,  Malikopoulos2021, Dave2022a}. The most frequently used information state for stochastic control is the \textit{belief state}, i.e., a distribution on the feasible states conditioned on the current memory \cite{varaiya_book}. However, in some systems with large state spaces, using an information state imposes severe computational challenges in DP. Thus, recent efforts focus on identifying \textit{approximate information states} \cite{kara2020near,subramanian2022approximate}. 
In contrast, for non-stochastic systems, DP has been simplified using a \textit{conditional range}, i.e., the set of state realizations consistent with the current memory \cite{bertsekas1973sufficiently, bernhard1996separation, bacsar2008h, rasouli2018scalable}. This concept has also been applied to decentralized systems \cite{gagrani2017decentralized, gagrani2020worst, Dave2021b}. However, to the best of our knowledge, there is no general definition of an information state or approximate information state for non-stochastic systems.

The contributions of this paper are (1) we introduce a general definition of an information state (Definition \ref{def_info_state}) along with a proof of the optimality of the resulting DP (Theorem \ref{thm_opt_dp}) and (2) we introduce an approximate information state for non-stochastic systems (Definition \ref{def_approx}), use it to formulate an approximate DP and present an upper bound on the resulting loss in performance (Theorems \ref{thm_approx_term_dp} - \ref{thm_approx_term_policy}). We illustrate our results using numerical simulations.

The remainder of the paper proceeds as follows. In Section \ref{section:problem}, we present our model. In Section \ref{section:info_state}, we define the notion of the information state and the corresponding DP. In Section \ref{section:approx}, we present the approximate information state, approximate DP, and bounds on the approximation loss. In Section \ref{section:example}, we present a numerical example to illustrate our results.
Finally, in Section \ref{section:conclusion}, we draw concluding remarks and discuss ongoing work.

\section{Model}
\label{section:problem}

\subsection{Notation and Preliminaries}

In this paper, we utilize the mathematical framework for \textit{uncertain variables}, which was presented in the context of non-stochastic information theory in \cite{nair2013nonstochastic}. 
An uncertain variable is a non-stochastic analogue of a random variable which take values in a known set and has an unknown distribution. Thus, for a sample space $\Omega$ and a given set $\mathcal{X}$, an uncertain variable is a mapping $X: \Omega \to \mathcal{X}$. For any $\omega \in \Omega$, the uncertain variable has the realization $X(\omega) = x \in \mathcal{X}$. 
The \textit{range} of an uncertain variable is analogous to the distribution of a random variable. The \textit{marginal range} of $X$ is the set $[[X]] \hspace{-1pt} := \hspace{-1pt} \{X(\omega) \; | \; \omega \in \Omega\}$. For two uncertain variables $X \in \mathcal{X}$ and $Y \in \mathcal{Y}$, their \textit{joint range} is the set $[[X,Y]] \hspace{-1pt} := \hspace{-1pt} \big\{ \big(X(\omega), Y(\omega) \big) \; | \; \omega \in \Omega \big\}$. For a given realization $y$ of $Y$, the \textit{conditional range} of $X$ is the set $[[X|y]] \hspace{-1pt} := \hspace{-1pt} \big\{ X(\omega) \; | \; Y(\omega) = y, \; \omega \in \Omega \big\}$ and in general, the conditional range of $X$ given $Y$ is $[[X|Y]] \hspace{-1pt} := \hspace{-1pt} \big\{ [[X|y]]\;| \; y \in [[Y]] \big\}$. 

Next, consider that the feasible sets $\mathcal{X}, \mathcal{Y}$ are compact, nonempty subsets of a metric space $(\mathcal{S},d)$, where $d(x, y)$ is the distance between any feasible realization $x \in \mathcal{X}$ of the uncertain variable $X$ and $y \in \mathcal{Y}$ of the uncertain variable $Y$. Furthermore, the distance between the sets $\mathcal{X}, \mathcal{Y}$ is measured using the Hausdorff metric \cite[Chapter 1.12]{barnsley2006superfractals}
\begin{align} \label{H_met_def}
    \mathcal{H}(\mathcal{X}, \mathcal{Y}) := \max \{ \max_{x \in \mathcal{X}} \min_{y \in \mathcal{Y}} d(x, y), \max_{y \in \mathcal{Y}} \min_{x \in \mathcal{X}} d(x, y) \}.
\end{align}

\subsection{Problem Formulation}

We consider a partially observed system where an agent selects control actions over $T \in \mathbb{N}$ discrete time steps. At any time $t=0,\dots, T$, the state of the system is denoted by an uncertain variable $X_t$ which takes values in a finite set $\mathcal{X}_t$.
The action at time $t$ is denoted by an uncertain variable $U_t$ which takes values in a finite set $\mathcal{U}_t$. Starting with the initial state $X_0 \in \mathcal{X}_0$, the state evolves as
$X_{t+1}=f_t\left(X_t,U_t,W_t\right)$ for all $t=0,\dots,T-1$. The uncertain variable $W_t$ denotes an uncontrolled disturbance acting on the state at time $t$, which takes values in a finite set $\mathcal{W}_t$. At each $t = 0,\dots,T$, the agent partially observes the state as an uncertain variable $Y_t=h_t(X_t,N_t)$ which takes values in a finite set $\mathcal{Y}_t$. The uncertain variable $N_t$ 
denotes a measurement noise which takes values in a finite set $\mathcal{N}_t$.
The external disturbances $\{W_t: t=0,\dots,T\}$, noises in measurement $\{N_t: t=0,\dots,T\}$, and initial state $X_0$ are collectively called \textit{primitive variables}. We consider that each primitive variable is independent of each other. This ensures that the system evolution is Markovian in a non-stochastic sense \cite{nair2013nonstochastic, bertsekas1973sufficiently}. We also consider that the agent has perfect memory and thus, at each $t=0,\dots,T$, the agent's memory is the set of uncertain variables $M_t := \{Y_{0:t}, U_{0:t-1}\}$ which takes values in a collection of sets $\mathcal{M}_t$. Note that $Y_{0:t} = \{Y_0,\dots,Y_t\}$.
After updating their memory, at each $t$ the agent selects the action
$U_t = g_t(M_t)$ using the control law $g_t: \mathcal{M}_t \to \mathcal{U}_t$. We denote the control strategy by $\boldsymbol{g} := (g_0,\dots,g_T)$ and the set of all feasible control strategies by $\mathcal{G}$. After $T$ time steps, the agent incurs a terminal cost $c_T(X_T, U_T) \in \mathbb{R}_{\geq0}$. We measure the system's performance by the \textit{maximum terminal cost} 
\begin{align} \label{eq_terminal_criterion}
\mathcal{J}(\boldsymbol{g}) := \max_{X_0, W_{0:T}, N_{0:T}} c_T(X_T, U_T).
\end{align}

\begin{problem} \label{problem_1}
The optimization problem is $\min_{\boldsymbol{g} \in \mathcal{G}} \mathcal{J}(\boldsymbol{g}),$
given the feasible sets $\{\mathcal{X}_0, \mathcal{W}_{t}, \mathcal{N}_{t} : t = 0,\dots,T\}$, the system dynamics $\{f_t, h_t: t=0,\dots,T\}$ and the terminal cost function $c_T$.
\end{problem}

Our aim is to tractably compute an optimal control strategy $\boldsymbol{g}^* \in \mathcal{G}$ for Problem \ref{problem_1}. This strategy is guaranteed to exist because all variables take values in finite sets. In our modeling framework, we impose the following assumptions:

\begin{assumption} \label{assumption_1}
We consider that the feasible sets $\{\mathcal{X}_t, \mathcal{Y}_t : t=0,\dots,T\}$ are both subsets of metric spaces.
\end{assumption}

Assumption \ref{assumption_1} allows us to measure the distance between any two realizations of $X_t$ and $Y_t$, respectively, for all $t$. To this end, we denote a generic metric by $d(\cdot, \cdot)$.


\begin{assumption} \label{assumption_2}

We consider that all uncertain variables at each $t$ and the cost $c_T(X_T, U_T)$ have a finite maximum value.

\end{assumption}

Since all feasible sets are finite, Assumption \ref{assumption_2} ensures that the functions $\{f_t, h_t : t=0,\dots,T\}$ and $c_T$ are globally Lipschitz. To this end, we will denote the Lipschitz constant of a function $f_t$ by $L_{f_t} \in \mathbb{R}_{\geq0}$.

\begin{remark}
While we derive our results for terminal cost problems, we also present an extension of our results to additive cost problems in Subsection \ref{subsection:info_examples}.
\end{remark}

\section{Dynamic Programs and Information States} \label{section:info_state}

In this section, we present a DP to derive the optimal control strategy for Problem \ref{problem_1}, and then define an information state which can simplify it. For all $t$, for each possible realization $m_t \in \mathcal{M}_t$ and $u_t \in \mathcal{U}_t$ of the memory $M_t$ and action $U_t$, respectively, we recursively define the \textit{value functions} of the DP at $t=0,\dots,T-1$ as
\begin{align}
    Q_t(m_t, u_t) := &\max_{m_{t+1} \in [[M_{t+1}|m_t, u_t]]} V_{t+1}(m_{t+1}), \label{DP_basic_1}\\
    V_t(m_t) := &\min_{u_t \in \mathcal{U}_t} Q_t(m_t, u_t), \label{DP_basic_2}
\end{align}
and $Q_T(m_{T},u_T) := \max_{x_T \in [[X_T|m_T]]}$ $ c_T(x_T,u_T)$ and $V_{T}(m_{T}) := \min_{u_T \in \mathcal{U}_T} Q_T(m_{T},u_T)$ at time $T$.
The control law at each $t$ is $g_t^*(m_t) := \arg \min_{u_t \in \mathcal{U}_t}$ $Q_t(m_t, u_t)$. 
The control strategy $\boldsymbol{g}^* = (g_0^*,\dots,$ $g_T^*)$ can be shown to be the optimal solution to Problem \ref{problem_1} using standard arguments \cite{bernhard2003minimax}. Note that at each $t$, the optimization in the RHS of \eqref{DP_basic_2} must be solved for each possible $m_t \in \mathcal{M}_t$. As $t$ increases, the size of the $\mathcal{M}_t$ increases with the addition of new information. Subsequently, a large number of computations are required to solve the DP for a long time horizon $T$. This issue motivates us to seek an uncertain variable, called an information state, which can be used in the DP instead of the memory, without loss of optimality.

\subsection{Information States}

In this subsection, we define an information state which is sufficient to construct a DP and prove that it yields an optimal control strategy.

\begin{definition} \label{def_info_state}
An \textit{information state} for Problem \ref{problem_1} at each $t=0,\dots,T$ is an uncertain variable $\Pi_t= \sigma_t(M_t)$ taking values in a finite set $\mathcal{P}_t$ and generated by $\sigma_t: \mathcal{M}_t \to \mathcal{P}_t$. Furthermore, for all $m_t \in \mathcal{M}_t$ and $u_t \in \mathcal{U}_t$, and for all $t=0,\dots,T$, it satisfies the following properties:

1) \textit{Sufficient to evaluate terminal cost:}
\begin{gather}
    \hspace{-10pt} \max_{x_T \in [[X_T|m_T]]} c_T(x_T, u_T) = \hspace{-5pt} \max_{x_T \in [[X_T|\sigma_T(m_T)]]} \hspace{-5pt} c_T(x_T, u_T). \label{p1}
\end{gather}

2) \textit{Sufficient to predict itself:}
\begin{gather}
    [[\Pi_{t+1}|m_t, u_t]] = [[\Pi_{t+1}|\sigma_t(m_t), u_t]]. \label{p2}
\end{gather}
\end{definition}

In the corresponding DP, for all $t$, and for all  $\pi_t \in \mathcal{P}_t$ and $u_t \in \mathcal{U}_t$, we recursively define the value functions at $t=0,\dots,T-1$ as
\begin{align}
    \bar{Q}_t(\pi_t, u_t) := &\max_{\pi_{t+1} \in [[\Pi_{t+1}|\pi_t, u_t]]} \bar{V}_{t+1}(\pi_{t+1}), \label{DP_info_1}\\
    \bar{V}_t(\pi_t) := &\min_{u_t \in \mathcal{U}_t} \bar{Q}_t(\pi_t, u_t), \label{DP_info_2}
\end{align}
and $\bar{Q}_T(\pi_{T},u_T) := \max_{x_T \in [[X_T|\pi_T]]}$ $ c_T(x_T,u_T)$ and $V_{T}(\pi_{T}) := \min_{u_T \in \mathcal{U}_T} Q_T(\pi_{T},u_T)$  at time $T$. The control law at each $t$ is $g_t^*(\pi_t) := \arg \min_{u_t \in \mathcal{U}_t} \bar{Q}_t(\pi_t, u_t)$. Next, we prove that the DP decomposition with information states \eqref{DP_info_1} - \eqref{DP_info_2} yields the same optimal value as the DP decomposition in \eqref{DP_basic_1} - \eqref{DP_basic_2} that uses the system's memory at each $t=0,\dots,T$. 

\begin{theorem} \label{thm_opt_dp}
Let $\Pi_t = \sigma_t(M_t)$ be an information state for Problem \ref{problem_1} for all $t=0,\dots,T$. Then, for all $m_t \in \mathcal{M}_t$ and $u_t \in \mathcal{U}_t$, ${Q}_t(m_t, u_t) = \bar{Q}_t(\sigma_t(m_t), u_t)$ and $V_t(m_t) = \bar{V}_t(\sigma_t(m_t))$.
\end{theorem}

\begin{proof}
Let $m_t \in \mathcal{M}_t$ and $u_t \in \mathcal{U}_t$ be given realizations of $M_t$ and $U_t$, respectively, for all $t$. We prove the result using mathematical induction, starting with $T$ where $Q_{T}(m_{T}, u_T) = \max_{x_T \in [[X_T|m_T]]} c_T(x_T, u_T) = \max_{x_T \in [[X_T|\sigma_T(m_T)]]} c_T(x_T, u_T) = \bar{Q}_{T}(\sigma_{T}(m_{T}), u_T)$ holds as a direct result of \eqref{p1} in Definition \ref{def_info_state}. Subsequently, by taking the minimum on both sides with respect to $u_t \in \mathcal{U}_t$, it holds that $V_T(m_T) = \bar{V}_T(\sigma_T(m_T))$. With this as the basis, for any $t=0,\dots,T-1$, we consider the induction hypothesis $V_{t+1}(m_{t+1}) = \bar{V}_{t+1}(\sigma_{t+1}(m_{t+1}))$. Next, we first prove that ${Q}_t(m_t, u_t) = \bar{Q}_t(\sigma_t(m_t), u_t)$ at $t$ by showing that the RHS of \eqref{DP_basic_1} is equal to the RHS of \eqref{DP_info_1}. Using the induction hypothesis in the RHS of \eqref{DP_basic_1},
$\max_{ m_{t+1} \in [[M_{t+1}|m_t, u_t]]}V_{t+1}(m_{t+1})
 = \max_{m_{t+1} \in [[M_{t+1}|m_t, u_t]]}\bar{V}_{t+1}(\sigma_{t+1}(m_{t+1}))
 = \max_{\sigma_{t+1}(m_{t+1}) \in [[\Pi_{t+1}|\sigma_t(m_t), u_t]]} \bar{V}_{t+1}(\sigma_{t+1}(m_{t+1})),$
where, in the second equality, we use the fact that $[[\Pi_{t+1}|m_t, u_t]] = \big\{\sigma_{t+1}(m_{t+1}) \in \mathcal{P}_{t+1}\big| m_{t+1} \in [[M_{t+1}|$ $m_t,u_t]]\big\}$ and \eqref{p2}. Thus, at time $t$, it holds that $Q_t(m_t, u_t) = \bar{Q}_t(\sigma_t(m_t), u_t)$. Subsequently, we can prove $V_t(m_t) = \bar{V}_t(\sigma_t(m_t))$ by minimizing both sides with $u_t \in \mathcal{U}_t$. This proves the induction hypothesis at time $t$, and the result follows by mathematical induction.
\end{proof}

Theorem \ref{thm_opt_dp} implies that the control strategy computed using \eqref{DP_info_1} - \eqref{DP_info_2} is an optimal solution to Problem \ref{problem_1}. In practice, using the information state in the DP decomposition only improves computational tractability when the set $\mathcal{P}_t$ has fewer elements than $\mathcal{M}_t$ for most time steps. This is usually true for systems with long time horizons. 

\subsection{Examples of Information States} \label{subsection:info_examples}

In this subsection, we present examples of information states which satisfy the conditions in Definition \ref{def_info_state}.

\textit{1) Perfectly Observed Systems:} Consider a system where $Y_t = X_t$ for all $t=0,\dots,T$. Then, the state is a valid information state \cite{bernhard2003minimax}, i.e., $\Pi_t = X_t$, which takes values in the set $\mathcal{X}_t$  and satisfies \eqref{p1} - \eqref{p2} for all $t$.

\textit{2) Partially Observed Systems:} The conditional range $\Pi_t = [[X_t|M_t]]$ is an information state for each $t=0,\dots,T$ in any partially observed system \cite{bernhard2003minimax}. This is a set-valued uncertain variable which takes values in the power set $2^{\mathcal{X}_t}$. Explicitly, for a given realization of the memory $m_t \in \mathcal{M}_t$ at time $t$, the conditional range takes the realization $P_t := \{ x_t \in \mathcal{X}_t | \exists x_0 \in \mathcal{X}_0, w_{0:t-1} \in \prod_{\ell = 0}^{t-1} \mathcal{W}_\ell, n_{0:t} \in \prod_{\ell = 0}^{t} \mathcal{N}_\ell \text{ such that } y_{t} = h_{t}(x_{t}, n_{t}), x_{\ell+1} = f_{\ell}(x_{\ell}, u_{\ell},w_{\ell}),$ $y_{\ell} = h_{\ell}(x_{\ell}, n_{\ell}) \text{ for all } \ell = 0, \dots, t-1 \}.$ We denote the realization by $P_t$ instead of $\pi_t$ to highlight that it is a set.

\textit{3) Additive Cost Problems:} Consider a partially observed system where the agent incurs a cost $c_t(X_t, U_t) \in \mathbb{R}_{\geq0}$ at each $t=0,\dots,T$ and the performance is measured by an additive performance criterion $\mathcal{J}^\text{ad}(\boldsymbol{g}) := \max_{X_0, W_{0:T}, N_{0:T}} \sum_{t=0}^T c_t(X_t, U_t).$
We can construct a DP and an information state for an additive cost problem by recasting it as a terminal cost problem \cite{bertsekas1973sufficiently}. At $t=0$, we define $A_0 := 0$ and for all $t=1,\dots,T$, we recursively define an uncertain variable $A_t \in \mathcal{A}_t$ as $A_t := A_{t-1} + c_{t-1}(X_{t-1}, U_{t-1})$. 
Then, at each $t$, we consider an augmented state $S_t = (X_t, A_t)$ and note that it yields a terminal cost $A_T + c_T(X_T,U_T)$. Thus, we can derive the optimal control strategy using the terminal cost DP. The information state at time $t$ is the conditional range $\Pi_t = [[X_t, A_t | M_t]]$ which takes values in the power set $2^{\mathcal{X}_t\times \mathcal{A}_t}$.

\begin{remark}
While the conditional range is an information state for all partially observed systems, the more general conditions in Definition \ref{def_info_state} can enable us to identify simpler information states for cases like systems with perfect observation. However, in many applications we may seek to construct an information state using limited data or seek to improve computational tractability of the DP by approximation. To account for these applications, in Section \ref{section:approx}, we extend Definition \ref{def_info_state} to define approximate information states.
\end{remark}

\section{Approximate Information State} \label{section:approx}

In this section, we define an approximate information state by relaxing the conditions in Definition \ref{def_info_state}. Then, we use it to develop an approximate DP and derive an upper bound on the resulting loss of optimality.

\begin{definition} \label{def_approx}
An \textit{approximate information state} for Problem \ref{problem_1}  is an uncertain variable $\hat{\Pi}_t = \hat{\sigma}_t(M_t)$, at each $t=0,\dots,T$, taking values in a finite set $\hat{\mathcal{P}}_t$ and generated by $\hat{\sigma}_t : \mathcal{M}_t \to \hat{\mathcal{P}}_t$. Furthermore, for all $t$, there exist parameters $\epsilon_T, \delta_t \in \mathbb{R}_{\geq0}$ such that for all $m_t \in \mathcal{M}_t$ and $u_t \in \mathcal{U}_t$ it is:

\textit{1) Sufficient to approximate terminal cost:}
\begin{multline}
    \hspace{-8pt} | \hspace{-3pt} \max_{x_T \in [[X_T|m_T]]} \hspace{-8pt} c_T(x_T, u_T) - \hspace{-8pt} \max_{x_T \in [[X_T|\hat{\sigma}_T(m_T)]]} \hspace{-4pt}  c_T(x_T, u_T)| \\
    \leq \epsilon_T. \hspace{-2pt} \label{ap1}
\end{multline}

\textit{2) Sufficient to approximate evolution:} We define $\mathcal{K}_{t+1} := [[ \hat{\Pi}_{t+1} ~|~ m_t, u_t]]$ and
    $\hat{\mathcal{K}}_{t+1} := [[ \hat{\Pi}_{t+1} ~|~ \hat{\sigma}_t(m_t), u_t ]]$. Then,
\begin{gather}
    \mathcal{H}(\mathcal{K}_{t+1}, \hat{\mathcal{K}}_{t+1}) \leq \delta_t, \label{ap2}
\end{gather}
where recall that $\mathcal{H}(\cdot)$ is the Hausdorff metric from \eqref{H_met_def}.
\end{definition}

In the approximate DP, for all $t=0,\dots,T-1$, for all $\hat{\pi}_t \in \hat{\mathcal{P}}_t$ and $u_t \in \mathcal{U}_t$, we recursively define the value functions
\begin{align}
    \hat{Q}_t(\hat{\pi}_t, u_t) := &\max_{\hat{\pi}_{t+1} \in [[\hat{\Pi}_{t+1}|\hat{\pi}_t, u_t]]} \label{DP_ap_term_1} \hat{V}_{t+1}(\hat{\pi}_{t+1}), \\
    \hat{V}_t(\hat{\pi}_t) := &\min_{u_t \in \mathcal{U}_t} \hat{Q}_t(\hat{\pi}_t, u_t), \label{DP_ap_term_2}
\end{align}
and $\hat{Q}_T(\hat{\pi}_{T},u_T) := \max_{x_T \in [[X_T|\hat{\pi}_T]]} c_T(x_T,u_T)$ and $\hat{V}_{T}($ $\hat{\pi}_{T}):= \min_{u_T \in \mathcal{U}_T} \hat{Q}_T(\hat{\pi}_{T},u_T)$ at time $T$. The control law at each $t$ is $\hat{g}_t^*(\hat{\pi}_t) := \arg \min_{u_t \in \mathcal{U}_t} \hat{Q}_t(\hat{\pi}_t, u_t)$ and the approximately optimal strategy is $\boldsymbol{\hat{g}}^* = (\hat{g}^*_0, \dots, \hat{g}_T^*)$. Next, in Theorem \ref{thm_approx_term_dp}, we establish an error bound when the value functions for the optimal DP \eqref{DP_basic_1} - \eqref{DP_basic_2} are approximated by \eqref{DP_ap_term_1} - \eqref{DP_ap_term_2}. We begin with two preliminary lemmas.


\begin{lemma} \label{lem_prelim}
Consider a metric space $(\mathcal{A},d)$ and two finite subsets $\mathcal{A}, \mathcal{B} \subset \mathcal{X}$. Let $f: \mathcal{X} \to \mathbb{R}$ be a function with a global Lipschitz constant $L_f \in \mathbb{R}_{\geq0}$. Then,
\begin{gather}
    \big| \max_{a \in \mathcal{A}} f(a) - \max_{b \in \mathcal{B}} f(b) \big| \leq L_f \cdot \mathcal{H}(\mathcal{A}, \mathcal{B}). \label{eq_prelim}
\end{gather}
\end{lemma}

\begin{proof}
We prove this result by considering two cases which are mutually exclusive but cover all the possibilities.
Case 1: $\max_{a \in \mathcal{A}} f(a) \geq \max_{b \in \mathcal{B}} f(b)$, which implies $| \max_{a \in \mathcal{A}} f(a) - \max_{b \in \mathcal{B}} f(b) | = \max_{a \in \mathcal{A}} f(a) - \max_{b \in \mathcal{B}} f(b) $. We define the non-empty set $\mathcal{A}^1 := \{a \in \mathcal{A} | f(a) \geq \max_{b \in \mathcal{B}} f(b)\}$ and note that $\max_{a \in \mathcal{A}} f(a) - \max_{b \in \mathcal{B}} f(b) = \max_{a \in \mathcal{A}^1} f(a) - \max_{b \in \mathcal{B}} f(b) = \max_{a \in \mathcal{A}^1} \min_{b \in \mathcal{B}}(f(a) - f(b)) \leq \max_{a \in \mathcal{A}} \min_{b \in \mathcal{B}}|f(a) - f(b)| \leq L_f \cdot \max_{a \in \mathcal{A}} \min_{b \in \mathcal{B}}|a -b|$. We can complete the proof for case 1 by invoking the definition of the Hausdorrf metric in \eqref{H_met_def} to conclude that
$| \max_{a \in \mathcal{A}} f(a) - \max_{b \in \mathcal{B}} f(b) | \leq L_f \cdot \max_{a \in \mathcal{A}} \min_{b \in \mathcal{B}}|a - b|
    \leq L_f \cdot \mathcal{H}(\mathcal{A}, \mathcal{B}).$
Case 2: $\max_{a \in \mathcal{A}} f(a) < \max_{b \in \mathcal{B}} f(b)$ and can prove the result using similar arguments as case 1. 
\end{proof}

\begin{lemma} \label{lem_prelim_2}
Consider a finite set $\mathcal{X}$ and two functions $f:\mathcal{X} \to \mathbb{R}$ and $g:\mathcal{X} \to \mathbb{R}$ with bounded outputs. Then,
\begin{align}
    |\max_{x \in \mathcal{X}} f(x) - \max_{x \in \mathcal{X}} g(x)| &\leq \max_{x \in \mathcal{X}}| f(x) - g(x)|, \label{prelim_2_1} \\
    |\min_{x \in \mathcal{X}} f(x) - \min_{x \in \mathcal{X}} g(x)| &\leq \max_{x \in \mathcal{X}}| f(x) - g(x)|. \label{prelim_2_2}
\end{align}
\end{lemma}

\begin{proof}
First, we prove \eqref{prelim_2_1} by considering two mutually exclusive cases which cover all possibilities. Case 1: We consider $\max_{x \in \mathcal{X}} f(x) \geq \max_{x \in \mathcal{X}} g(x)$, which implies $|\max_{x \in \mathcal{X}} f(x) - \max_{x \in \mathcal{X}} g(x)| = \max_{x \in \mathcal{X}} f(x) - \max_{x \in \mathcal{X}} g(x)$. Then, we define $x^* := \arg \max_{x \in \mathcal{X}} f(x)$ and note
$\max_{x \in \mathcal{X}} f(x) - \max_{x \in \mathcal{X}} g(x) = f(x^*) - \max_{x \in \mathcal{X}}$ $g(x) \leq f(x^*) - g(x^*) \leq \max_{x \in \mathcal{X}}|f(x) - g(x)|$. Case 2: $\max_{x \in \mathcal{X}} f(x) < \max_{x \in \mathcal{X}} g(x)$. The proof can be completed using similar arguments as in Case 1.
The proof for \eqref{prelim_2_2} follows from similar arguments as \eqref{prelim_2_1}. Due to space limitations, it is omitted.
\end{proof}

\begin{theorem} \label{thm_approx_term_dp}
Let $L_{\hat{V}_{t+1}}$ be the Lipschitz constant of $\hat{V}_{t+1}(\cdot)$ for all $t=0,\dots,T$. Then, for all $m_t \in \mathcal{M}_t$ and $u_t \in \mathcal{U}_t$:
\begin{align}
    |Q_t(m_t, u_t) - \hat{Q}_t(\hat{\sigma}_t(m_t), u_t)| \leq \alpha_t, \label{thm_2_1} \\
    |V_t(m_t) - \hat{V}_t(\hat{\sigma}_t(m_t))| \leq \alpha_t, \label{thm_2_2}
\end{align}
where  $\alpha_T = \epsilon_T$ and $\alpha_t = \alpha_{t+1} + L_{\hat{V}_{t+1}} \cdot \delta_t$ for all $t=0,\dots,T-1$.
\end{theorem}

\begin{proof} For all $t$, let $m_t \in \mathcal{M}_t$ and $u_t \in \mathcal{U}_t$ be the realizations of $M_t$ and $U_t$, respectively. We prove both results by mathematical induction, starting at time $T$, where \eqref{thm_2_1} follows from \eqref{ap1} in Definition \ref{def_approx}. For \eqref{thm_2_2}, we can expand the LHS as $|V_T(m_T) - \hat{V}_T(\hat{\sigma}_T(m_T))| = |\min_{u_T \in \mathcal{U}_T}Q_T(m_T, u_T) - \min_{u_t \in \mathcal{U}_T}\hat{Q}_T(\hat{\sigma}_T(m_T), u_T)| \leq \max_{u_T \in \mathcal{U}_T} |Q_T(m_T, u_T) - \hat{Q}_T(\hat{\sigma}_T(m_T), u_T)| \leq \epsilon_T$, where in the first inequality, we use \eqref{prelim_2_2} from Lemma \ref{lem_prelim_2} and in the second inequality we use \eqref{thm_2_1}. 
Next, for all $t$, we consider the hypothesis $|V_{t+1}(m_{t+1})- \hat{V}_{t+1}$ $(\hat{\sigma}_{t+1}(m_{t+1}))| \leq \alpha_{t+1}$. Then, $|Q_t(m_t, u_t) - \hat{Q}_t(\hat{\sigma}_t(m_t),u_t)|= \big|\max_{m_{t+1} \in [[M_{t+1}|m_t, u_t]]} V_{t+1}(m_{t+1}) - \max_{\hat{\pi}_{t+1} \in [[\hat{\Pi}_{t+1}|\hat{\sigma}_t(m_t), u_t]]}$ $\hat{V}_{t+1}(\hat{\pi}_{t+1})\big|$. Then,
\begin{align}
    &|Q_t(m_t, u_t) - \hat{Q}_t(\hat{\sigma}_t(m_t), u_t) | \leq \big| \max_{m_{t+1} \in [[M_{t+1}|m_t, u_t]]}  \nonumber \\
    & V_{t+1}(m_{t+1}) - \hspace{-5pt} \max_{\hat{\sigma}_{t+1}(m_{t+1}) \in [[\hat{\Pi}_{t+1}|m_t, u_t]]} \hspace{-7pt} \hat{V}_{t+1}(\hat{\sigma}_{t+1}(m_{t+1})) \big| + \nonumber \\
    &\big|\hspace{-2pt} \max_{\hat{\pi}_{t+1} \in [[\hat{\Pi}_{t+1}|m_t, u_t]]} \hspace{-12pt} \hat{V}_{t+1}(\hat{\pi}_{t+1}) 
    - \hspace{-7pt} \max_{\hat{\pi}_{t+1} \in [[\hat{\Pi}_{t+1}|\hat{\sigma}_t(m_t), u_t]]} \hspace{-19pt} \hat{V}_{t+1}(\hat{\pi}_{t+1}) \big|, \label{proof_2_1}
\end{align}
where, we use the triangle inequality.
In the first term, $\big| \max_{m_{t+1} \in [[M_{t+1}|m_t, u_t]]} V_{t+1}(m_{t+1}) \hspace{-2pt} - \hspace{-2pt} \max_{\hat{\sigma}_{t+1}(m_{t+1}) \in [[\hat{\Pi}_{t+1}|m_t, u_t]]}$ $\hat{V}_{t+1}(\hat{\sigma}_{t+1}(m_{t+1}))$ $\big| \leq \max_{m_{t+1} \in [[M_{t+1}|m_t, u_t]]} \big|V_{t+1}(m_{t+1})
     - \hat{V}_{t+1}(\hat{\sigma}_{t+1}(m_{t+1}))\big|$ $
    \leq \alpha_{t+1},$
where, in the first inequality, we note that $[[\hat{\Pi}_{t+1}|m_t, u_t]] = \{\hat{\sigma}_{t+1}(m_{t+1}) \in \hat{\mathcal{P}}_t | m_{t+1} \in [[M_{t+1}|m_t, u_t]] \}$ and use \eqref{prelim_2_1} from Lemma \ref{lem_prelim_2}; and, in the second inequality, we use the induction hypothesis. 
The second term in the RHS of \eqref{proof_2_1} satisfies $\big|\max_{\hat{\pi}_{t+1} \in [[\hat{\Pi}_{t+1}|m_t, u_t]]} \hat{V}_{t+1}(\hat{\pi}_{t+1}) 
    - \hspace{-5pt} \max_{\hat{\pi}_{t+1} \in [[\hat{\Pi}_{t+1}|\hat{\sigma}_t(m_t), u_t]]} \hat{V}_{t+1}(\hat{\pi}_{t+1}) \big| \leq L_{\hat{V}_{t+1}} \hspace{-4pt} \cdot \delta_t$ using \eqref{eq_prelim} from Lemma \ref{lem_prelim} and \eqref{ap2} from Definition \ref{def_approx}.
Substituting the inequality for each term in the RHS of  \eqref{proof_2_1} yields $|Q_t(m_t, u_t) - \hat{Q}_t(\hat{\sigma}_t(m_t), u_t) | \leq \alpha_{t+1} + L_{\hat{V}_{t+1}} \cdot \delta_t.$
Subsequently, we can prove \eqref{thm_2_2} as $V_t(m_t) - \hat{V}_t(\hat{\sigma}_t(m_t)) = |\min_{u_t \in \mathcal{U}_t}Q_t(m_t, u_t) - \min_{u_t \in \mathcal{U}_t}\hat{Q}_t(\hat{\sigma}_t(m_t), u_t) | \leq \max_{u_t \in \mathcal{U}_t}|Q_t(m_t, u_t) - \hat{Q}_t(\hat{\sigma}_t(m_t), u_t) | \leq \alpha_t$, where, in the first inequality, we use \eqref{prelim_2_2} from Lemma \ref{lem_prelim_2}. This completes our proof by induction for all $t=0,\dots,T$.
\end{proof}

After bounding the approximation error for value functions, we also seek to bound the optimality gap from using the approximate strategy. Consider the approximately optimal strategy $\boldsymbol{\hat{g}} := (\hat{g}_0, \dots, \hat{g}_T)$ with $\hat{g}_t(\hat{\pi}_t) = \arg \min_{u_t \in \mathcal{U}_t} \hat{Q}_t(\hat{\pi}_t, u_t)$ for all $t$. Then, the equivalent strategy $\boldsymbol{g}=(g_0,\dots,g_T)$ using memory has $g_t(m_t) := \hat{g}_t(\hat{\sigma}_t(m_t))$ for all $t$. To compute the performance of $\boldsymbol{g}$, we define for all $t=0,\dots,T-1$, for all $m_t \in \mathcal{M}_t$ and $u_t \in \mathcal{U}_t$:
\begin{align}
\Theta_t(m_t,u_t) &:=  \max_{m_{t+1} \in [[M_{t+1}|m_t,u_t]]} \Lambda_{t+1}(m_{t+1}), \\
\Lambda_t(m_t) &:= \Theta_t(m_t,g_t(m_t)), 
\end{align}
and $\Theta_T(m_T,u_T):= \max_{x_T \in [[X_T|m_T]]}$ $c_T(x_T,u_T)$ and $\Lambda_T(m_T) := \Theta_T(m_T,g_T(m_T))$  for time $T$. Then, because $m_0 = \{y_0\}$, the performance of $\boldsymbol{g}$ is $\Lambda_0(y_0)$ for any $y_0 \in \mathcal{Y}_0$. Next, we bound the difference in the performance of the approximate strategy $\boldsymbol{g}$ and the optimal strategy.

\begin{theorem} \label{thm_approx_term_policy}
Let $L_{\hat{V}_{t+1}}$ be the Lipschitz constant of $\hat{V}_{t+1}(\cdot)$ for all $t=0,\dots,T$. Then, for all $m_t \in \mathcal{M}_t$ and $u_t \in \mathcal{U}_t$,
\begin{align}
    |Q_t(m_t, u_t) - \Theta_t(m_t, u_t)| \leq 2\alpha_t, \label{thm_2_3} \\
    |V_t(m_t) - \Lambda_t(m_t)| \leq 2\alpha_t, \label{thm_2_4}
\end{align}
where  $\alpha_T = \epsilon_T$ and $\alpha_t = \alpha_{t+1} + L_{\hat{V}_{t+1}} \hspace{-4pt} \cdot \delta_t$ for all $t=0,\dots,T-1$.
\end{theorem}

\begin{proof}
For all $t=0,\dots,T-1$ and for each $\hat{\pi}_t \in \hat{\mathcal{P}}_t$ and $u_t \in \mathcal{U}_t$, let
$\hat{\Theta}_t$ $(\hat{\pi}_t, u_t)$ $ := \max_{\hat{\pi}_{t+1} \in [[\hat{\Pi}_{t+1}|\hat{\pi}_t, u_t]]}$ $ \hat{\Lambda}_{t+1}(\hat{\pi}_{t+1})$ and $\hat{\Lambda}_t(\hat{\pi}_t)$ $:= \hat{\Theta}_t(\hat{\pi}_t, \hat{g}_t(\hat{\pi}_t))$. At $t=T$, $\hat{\Theta}_T(\hat{\pi}_T, u_T) := \max_{x_T \in [[X_T|\hat{\pi}_T]]} c_T(x_T,u_T)$ and $\hat{\Lambda}_{T}$ $(\hat{\pi}_T):= \hat{\Theta}_T(\hat{\pi}_T, \hat{g}_T(\hat{\pi}_T))$.
Note that $\hat{\Theta}_t(\hat{\pi}_t, u_t) = \hat{Q}_{t}(\hat{\pi}_t, u_t)$ and $\hat{\Lambda}_t(\hat{\pi}_t) = \hat{V}_t(\hat{\pi}_t)$ for all $t$ since $\hat{g}_t(\hat{\pi}_t)  =\arg\min_{u_t \in \mathcal{U}_t}\hat{Q}_t(\hat{\pi}_t,u_t)$. Next, at each $t=0,\dots,T$, we use the triangle inequality in the LHS of \eqref{thm_2_3}:
\begin{align}
    |Q_t(m_t, u_t) - &\Theta_t(m_t, u_t)| \leq |Q_t(m_t, u_t) -  \nonumber \\
    \hat{Q}_t(\hat{\sigma}_t(m_t),u_t)|
     &+ |\hat{\Theta}_t(\hat{\sigma}_t(m_t), u_t) - \Theta_t(m_t, u_t)| \nonumber \\
    \leq \alpha_t  + &|\hat{\Theta}_t(\hat{\sigma}_t(m_t), u_t) - \Theta_t(m_t, u_t)|, \label{proof_2_3}
\end{align}
where, in the second inequality, we use \eqref{thm_2_1} from Theorem \ref{thm_approx_term_dp}. 
Next, we prove that $|\hat{\Theta}_t(\hat{\sigma}_t(m_t), u_t) - \Theta_t(m_t, u_t)| \leq \alpha_t$ and $|\hat{\Lambda}_t(\hat{\sigma}_t(m_t)) - \Lambda_t(m_t)| \leq \alpha_t$ for all $t=0,\dots,T$ using backward mathematical induction starting at time $T$.
At time $T$, from \eqref{ap2} in Definition \ref{def_approx}, $|\hat{\Theta}_{T}(\hat{\sigma}_T(m_T), u_T) \hspace{-2pt} - \hspace{-2pt} \Theta_T(m_T, u_T)| \hspace{-2pt} = \hspace{-2pt} |\max_{x_T \in [[X_T|\hat{\sigma}_T(m_T)]]}$ $c_T(x_T,u_T) - \max_{x_T \in [[X_T|m_T]]} c_T(x_T,u_T) | \leq \epsilon_T$. 
Furthermore, $|\hat{\Lambda}_{T}(\hat{\sigma}_T(m_T)) - \Lambda_T(m_T)| = |\hat{\Theta}_{T}(\hat{\sigma}_T(m_T), \hat{g}_T(\hat{\sigma}_T(m_T))) - \Theta_T(m_T, g_T(m_T))| \leq \epsilon_T$, where the inequality holds because $g_T(m_T) = \hat{g}_T(\hat{\sigma}_T(m_T))$. 
With this as the basis, for any $t=0,\dots,T-1$, we consider the induction hypothesis $|\hat{\Lambda}_{t+1}(\hat{\sigma}_{t+1}(m_{t+1})) - \Lambda_{t+1}(m_{t+1})| \leq \alpha_{t+1}$. Then, using the definitions of the value functions, $|\hat{\Theta}_{t}(\hat{\sigma}_t(m_t), u_t) - \Theta_t(m_t, u_t)| = | \max_{\hat{\pi}_{t+1}\in [[\hat{\Pi}_{t+1}|\hat{\sigma}_t(m_t),u_t]]}\hat{\Lambda}_{t+1}(\hat{\pi}_{t+1}) - \max_{m_{t+1} \in [[M_{t+1}|m_t,u_t]]}\Lambda_{t+1}(m_{t+1}) |$. Next, we expand the RHS using the triangle inequality to write that
\begin{multline}
    \hspace{-2pt}|\hat{\Theta}_{t}(\hat{\sigma}_t(m_t), u_t) - \Theta_t(m_t, u_t)| \leq \big| \max_{\hat{\pi}_{t+1}\in [[\hat{\Pi}_{t+1}|\hat{\sigma}_t(m_t),u_t]]} \\
    \hspace{-1pt} \hat{\Lambda}_{t+1}\hspace{-0.5pt}(\hat{\pi}_{t+1}) \hspace{-1pt} - \hspace{-10pt} \max_{\hat{\pi}_{t+1}\in [[\hat{\Pi}_{t+1}|m_t,u_t]]} \hspace{-9pt} \hat{\Lambda}_{t+1}\hspace{-0.5pt}(\hat{\pi}_{t+1})\hspace{-0.5pt}| + | \hspace{-3pt}\max_{m_{t+1} \in [[M_{t+1}|m_t,u_t]]} \\\Lambda_{t+1}(m_{t+1}) - \max_{\hat{\sigma}_{t+1}(m_{t+1})\in[[\hat{\Pi}_{t+1}|m_t,u_t]]} \hat{\Lambda}_{t+1}(\hat{\sigma}_{t+1}(m_{t+1}))| \\
    \leq L_{\hat{V}_{t+1}} \cdot \delta_t + \alpha_{t+1} = \alpha_t, \label{proof_2_4}
\end{multline}
where, in the second inequality, the first term is upper bounded by noting that $\hat{\Lambda}_{t+1}(\hat{\pi}_{t+1}) = \hat{V}_{t+1}(\hat{\pi}_{t+1})$, using \eqref{eq_prelim} from Lemma \ref{lem_prelim} and then using \eqref{ap2} from Definition \ref{def_approx}, whereas the second term is bounded using the induction hypothesis.
Following the same sequence of arguments as time $T$,  $|\hat{\Lambda}_{t}(\hat{\sigma}_t(m_t)) - \Lambda_t(m_t)| \leq \alpha_t$ as a direct consequence of \eqref{proof_2_4}. This completes the induction. The proof is complete by substituting this result into the RHS of \eqref{proof_2_3} to show \eqref{thm_2_3} and consequently, \eqref{thm_2_4}.
\end{proof}

\subsection{Examples} \label{subsection:approx_examples}

In this subsection, we present two approximate information states which satisfy Definition \ref{def_approx}. They are both inspired by state quantization \cite{bertsekas1975convergence}. Specifically, at any $t=0,\dots,T$, let $\mathcal{X}_t$ be the set of feasible states. Then, a subset $\hat{\mathcal{X}}_t \subset \mathcal{X}_t$ is a \textit{set of quantized states} with parameter $\gamma_t \in \mathbb{R}_{\geq0}$ if $\max_{x_t \in \mathcal{X}_t} \min_{\hat{x_t} \in \hat{\mathcal{X}}_t} d(x_t, \hat{x}_t) \leq \gamma_t$. The corresponding \textit{quantization function} $\mu_t: \mathcal{X}_t \to \hat{\mathcal{X}}_t$ is defined as $\mu_t(x_t) := \arg \min_{\hat{x}_t \in \hat{\mathcal{X}}_t} d(x_t, \hat{x}_t)$. Note that by construction, $d(x_t, \mu_t(x_t)) \leq \gamma_t$ for all $x_t \in \mathcal{X}_t$, for all $t$.

\textit{1) Perfectly Observed Systems:} Consider a system where $Y_t = X_t$ for all $t=0,\dots,T$. Recall from Subsection \ref{subsection:info_examples} that the information state is simply $\Pi_t=X_t$ and it takes values in $\mathcal{X}_t$ for all $t$. Then, an approximate information state for such a system is the quantized state $\hat{\Pi}_t := \mu_t(X_t)$ with $\epsilon_T = 2L_{c_T} \cdot \gamma_T$ and $\delta_t = 2 \gamma_{t+1} + 2  L_{f_t} \cdot \gamma_t$, where $\gamma_{T+1} = 0$, and $L_{c_T}$ and $L_{f_t}$ are the Lipschitz constants for $c_T(\cdot)$ and $f_t(\cdot)$, respectively. The derivation for the values of $\epsilon_T$ and $\delta_t$ can be found in Appendix A.



\textit{2) Partially Observed Systems:} For a partially observed system, recall from Section \ref{subsection:info_examples} that an information state is given by the conditional range $\Pi_t = [[X_t|m_t]]$. We approximate the conditional range by quantizing each element in $\Pi_t$. This is generated by the mapping $\nu_t: 2^{\mathcal{X}_t} \to 2^{\hat{\mathcal{X}}_t}$ which yields the approximate range $\nu_t(\Pi_t) := \{\mu_t(x_t) \in \hat{\mathcal{X}}_t|x_t \in \Pi_t\}.$ Then, $\hat{\Pi}_t = \nu_t(\Pi_t)$ is an approximate information state for partially observed systems for all $t=0,\dots,T$ with $\epsilon_T = 2L_{c_T} \cdot \gamma_T$ and $\delta_t = 2 \gamma_{t+1} + 2 L_{\bar{f}_t} \cdot L_{h_{t+1}} \cdot L_{f_t} \cdot \gamma_t$, where $\gamma_{T+1} = 0$, and $L_{c_T}$, $L_{\bar{f}_t}$, $L_{h_{t+1}}$ and $L_{f_t}$ are Lipschitz constants of $c_T(\cdot)$, $\bar{f}_t(\cdot)$, $h_{t+1}(\cdot)$, and $f_t(\cdot)$, respectively.  The derivation of the values of $\epsilon_T$ and $\delta_t$ can be found in Appendix B.

\section{Numerical Example} \label{section:example}

In this section, we present a numerical example illustrating the performance of the approximate conditional range for a gridworld pursuit problem. We consider an agent who seeks to catch a moving target at the end of a time horizon $T$ on $9 \times 9$ grid with static obstacles.
For all $t=0,\dots,T$, we denote the position of the agent by $X_t^\text{ag}$ and the position of the target by $X_t^{\text{ta}}$, each of which takes values in the set of grid cells $\mathcal{X} = \big\{(-4,-4),(-4,-3),\dots,(3,4),(4,4)\big\} \setminus \mathcal{O}$, where $\mathcal{O} \subset \mathcal{X}$ is the set of obstacle cells. Let $\mathcal{U}_t = \mathcal{W}_t = \mathcal{N}_t = \{(-1,0),(1,0),(0,0),(0,1),$ $(0,-1)\}$ for all $t$. Then, starting at $X_0^{\text{ta}} \in \mathcal{X}$, the target's position is updated as $X_{t+1}^{\text{ta}} = \mathbb{I}(X_t^{\text{ta}} + W_t \in \mathcal{X})\cdot(X_t^{\text{ta}} + W_t) + (1 - \mathbb{I}(X_t^{\text{ta}} + W_t \in \mathcal{X}) )\cdot X_t^{\text{ta}}$, where $W_t \in \mathcal{W}_t$ and $\mathbb{I}(\cdot)$ is the indicator function. 
At each $t$, the agent receives an observation of the target $Y_t = \mathbb{I}(X_t^{\text{ta}} + N_t \in \mathcal{X})\cdot(X_t^{\text{ta}} + N_t) + (1-\mathbb{I}(X_t^{\text{ta}} + N_t \in \mathcal{X}) )\cdot X_t^{\text{ta}}$, where $N_t \in \mathcal{N}_t$. The agent observes their own position perfectly. Then, the agent selects an action $U_t \in \mathcal{U}_t$ and updates their position to $X_{t+1}^{\text{ag}} = \mathbb{I}(X_t^{\text{ag}} + U_t \in \mathcal{X})\cdot(X_t^{\text{ag}} + U_t) + (1 - \mathbb{I}(X_t^{\text{ag}} + U_t \in \mathcal{X}) )\cdot X_t^{\text{ag}}$.
The terminal cost after $T$ time steps is $d(X_T^{\text{ta}}, X_T^{\text{ag}})$, where $d(\cdot,\cdot)$ is the shortest path between two cells while avoiding all obstacles. The distance between two adjacent cells is $1$ unit.
The grid and an initial state are illustrated in Fig. \ref{fig:1a}. Here, the black colored cells mark the obstacles. The red triangle and the red circle indicate the position and observation of the agent, respectively, at $t=0$. The red hatched region indicates the possible locations of the target at $t=0$.



\begin{figure}[ht!]
  \centering
  \subfigure[The original grid \label{fig:1a}]{\includegraphics[width=0.30\linewidth, keepaspectratio]{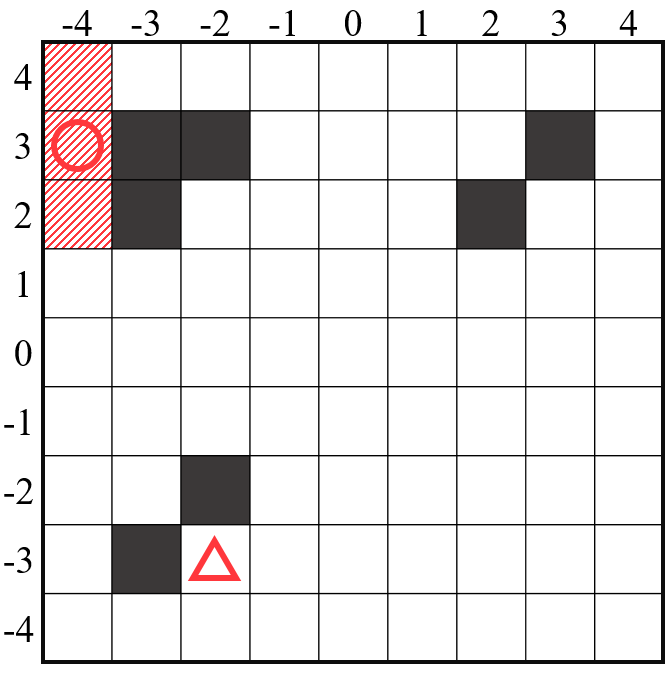}} 
  \subfigure[The quantized grid \label{fig:1b}]{\includegraphics[width=0.30\linewidth, keepaspectratio]{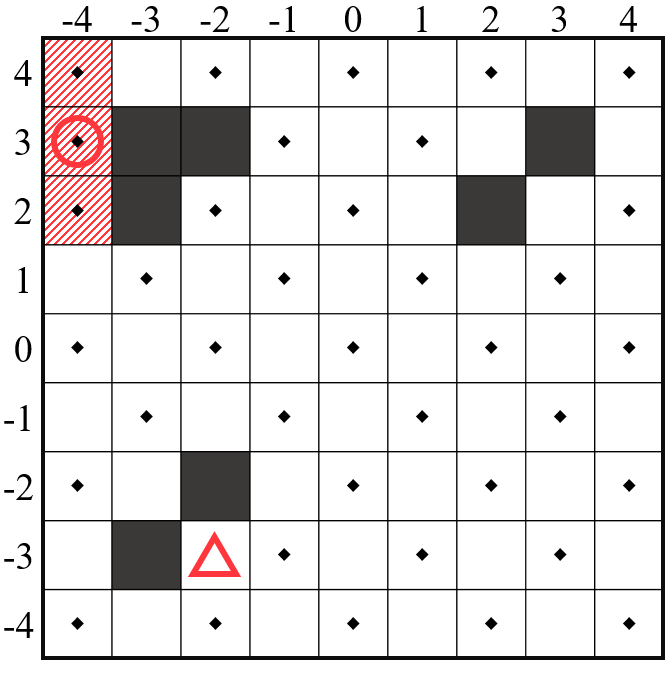}}
  \vspace{-10pt}
  \caption{The gridworld pursuit problem with the initial conditions $x_0^{\text{ag}} = (-2,-3)$ and $y_0 = (-4,3)$.}
  \label{fig:illustration}
  \vspace{-8pt}
\end{figure}

Recall from Subsection \ref{subsection:info_examples} that an information state at time $t$ is $\Pi_t = \big(X_t^{\text{ag}}, [[X_t^{\text{ta}}|M_t]]\big) \in \mathcal{X} \times 2^{\mathcal{\mathcal{X}}}$. We construct an approximation of the conditional range $[[X_t^{\text{ta}}|M_t]]$ at time $t$ using the quantization approach from Subsection \ref{subsection:approx_examples}. The set of quantized states $\hat{\mathcal{X}}$, with $\gamma_t = 1$ for all $t$, is illustrated in Fig. \ref{fig:1b} by marking the relevant cells with dots. Recall that $\mu_t(x_t) = \arg\min_{\hat{x}_t \in \hat{\mathcal{X}}}d(x_t,\hat{x}_t)$ and the approximate range at time $t$ is $\hat{A}_t  = \big\{ \mu_t(x_t) \in \hat{\mathcal{X}} | x_t \in [[X_t^{\text{ta}}|M_t]] \big\}$. We consider the approximate information state $\hat{\Pi}_t = \big(X_t^{\text{ag}}, \hat{A}_t, Y_0\big) \in \mathcal{X} \times 2^{\hat{\mathcal{X}}} \times \mathcal{X}$ for all $t$.
The initial observation $Y_0$ in $\hat{\Pi}_t$ makes the prediction of $\hat{A}_{t+1}$ more accurate. For five initial conditions, we compute the best control strategy for $T=6$ using both the information state (IS) and the approximate information state (AIS). In Fig. \ref{fig:table}, we present the worst-case costs for both the DPs ($V_0$ and $\hat{V}_0$) and the computational times (Run.) in seconds. 
Note that the approximate DP has a significantly faster runtime. We also implement both strategies on the system with random disturbances. 
In Fig. \ref{fig:table}, we illustrate differences in \textit{actual} costs across 1000 such simulations when implementing the approximate strategy and the optimal strategy. The dots indicate the most frequently observed difference for each case. Note that the difference in costs is bounded.




\begin{figure}[ht!]
  \centering
  \includegraphics[width=\linewidth, keepaspectratio]{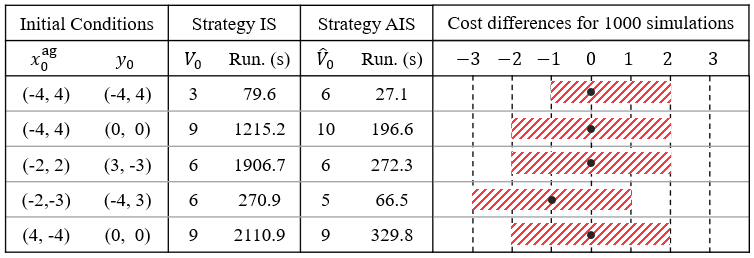}
  \caption{Results of numerical simulations for $T=6$.}
  \label{fig:table}
  \vspace{-12pt}
\end{figure}

\section{Conclusion} \label{section:conclusion}

In this paper, we presented a theory of information states for non-stochastic control problems. We characterized the information states by their properties and proved that they can be used to derive the optimal control strategy. Then, we proposed a definition for an approximate information state which yielded an approximate DP. We provide upper bounds on the approximation error if an agent uses the approximately optimal strategy to generate their control actions.
Future work should consider constructing approximate information states using only partial knowledge of system dynamics. 

\bibliographystyle{ieeetr}
\bibliography{References}

\section*{Appendix A - Derivation of Approximate Information State for Perfectly Observed Systems}

In this appendix, we derive the values of $\epsilon_T$ and $\delta_t$ for all $t=0,\dots,T$ when an approximate information state is constructed using state quantization for a perfectly observed system. We first state a property of the Hausdorff metric which we will use in our derivation.

\begin{lemma} \label{lem_union_prop}
Let $\mathcal{X}$ be a metric space with finite subsets
$\mathcal{A}, \mathcal{B}, \mathcal{C}, \mathcal{D} \subset \mathcal{X}$. Then, it holds that
\begin{align}
    \mathcal{H}\big(\mathcal{A} \cup \mathcal{B}, \mathcal{C} \cup \mathcal{D}\big)
    \leq \max \Big\{  \mathcal{H}\big(\mathcal{A}, \mathcal{C}\big) , \mathcal{H}\big(\mathcal{B}, \mathcal{D} \big) \Big\}. \label{eq_union_prop}
\end{align}
\end{lemma}

\begin{proof}
The proof for this result is given in \cite[Theorem 1.12.15]{barnsley2006superfractals}.
\end{proof}

Next, recall from Subsection \ref{subsection:approx_examples} that for all $t=0,\dots,T$, the set of quantized states is $\hat{\mathcal{X}}_t \subseteq \mathcal{X}_t$ with a parameter $\gamma_t$. For any state $x_t \in \mathcal{X}_t$ the corresponding quantized state is $\mu_t(x_t) = \arg \min_{\hat{x}_t \in \hat{\mathcal{X}}_t} d(x_t, \hat{x}_t)$ and it holds that $\max_{x_t \in \mathcal{X}_t} d(x_t, \mu_t(x_t)) \leq \gamma_t$. Next, we state and prove the main result of this appendix.

\begin{theorem}
For all $t=0,\dots,T$, let $\mu_t: \mathcal{X}_t \to \hat{\mathcal{X}}_t$ such that $\max_{x_t \in \mathcal{X}_t} d(x_t, \mu_t(x_t)) \leq \gamma_t$ and let $L_{c_T}$, and $L_{f_t}$ be Lipschitz constants for $c_T(\cdot)$ and $f_t(\cdot)$, respectively. Then, for all $t$, $\hat{\Pi}_t = \mu_t(X_t)$ is an approximate information state which satisfies \eqref{ap1} with $\epsilon_T = 2L_{c_T} \cdot \gamma_T$ and \eqref{ap2} with $\delta_t = 2 \gamma_{t+1} + 2 L_{f_t} \cdot \gamma_t$, where $\gamma_{T+1} = 0$.
\end{theorem}

\begin{proof}
For all $t=0,\dots,T$, let $m_t = \{x_{0:t}, u_{0:t-1}\}$ be the realization of $M_t$ and let the approximate information state be $\hat{x}_t = \mu_t(x_t)$. We first derive the value of $\epsilon_t$ in the RHS of \eqref{ap1}. At time $t$, can expand the conditional ranges to write that $[[X_t|m_t]] = [[X_t|x_t]] = \{x_t\}$ and $[[X_t|\hat{x}_t]] = \{x_t \in \mathcal{X} : d(x_t, \hat{x}_t) \leq \gamma_t \}$. 
On substituting these into the LHS of \eqref{ap1}, we state that
\begin{align}
    &\big|\max_{\bar{x}_T \in [[X_T|m_T]]} c_T(\bar{x}_T, u_T) - \max_{\bar{x}_T \in [[X_T|\mu_T(x_T)]]} c_T(\bar{x}_T,u_T)\big| \nonumber \\
    = &\big|c_T(x_T,u_T) - \max_{\bar{x}_T \in [[X_T|\mu_T(x_T)]]}c_T(\bar{x}_T,u_T)\big| \nonumber \\
    \leq &\max_{\bar{x}_T \in [[X_T|\mu_T(x_T)]]} |c_T(x_T,u_T) - c_T(\bar{x}_T,u_T)| \nonumber \\
    \leq &L_{c_T} \cdot \max_{\bar{x}_T \in [[X_T|\mu_T(x_T)]]} d(x_T, \bar{x}_T) \nonumber \\
    \leq &L_{c_T}\cdot\Big(d(x_T, \mu_T(x_T)) + \max_{\bar{x}_T \in [[X_T|\mu_T(x_T)]]}d(\mu_T(x_T), \bar{x}_T)\Big), \nonumber \\
    \leq &2 L_{c_T} \cdot \gamma_T =: \epsilon_T,
\end{align}
where, in the third inequality, we use the triangle inequality.

Next, to derive the value of $\delta_t$ in the RHS of \eqref{ap2}, we expand the LHS of \eqref{ap2} as
\begin{align}
    &\mathcal{H}\big([[\hat{X}_{t+1}|x_t, u_t]],[[\hat{X}_{t+1}|\mu_t(x_t),u_t]]\big) \nonumber \\
    = &\mathcal{H}\Big( \big\{ \mu_{t+1}(f_t(x_t,u_t,w_t))| w_t \in \mathcal{W}_t \big\},  \nonumber \\
    &\big\{ \mu_{t+1}(f_t(\bar{x}_t,u_t,w_t))| \bar{x}_t \in [[X_t|\mu_t(x_t)]], w_t \in \mathcal{W}_t \big\}\Big) \nonumber \\
    \leq &\max_{w_t \in \mathcal{W}_t} \mathcal{H}(\{ \mu_{t+1}(f_t(x_t,u_t,w_t))\},  \nonumber \\
    &\{ \mu_{t+1}(f_t(\bar{x}_t,u_t,w_t))| \bar{x}_t \in [[X_t|\mu_t(x_t)]]\}), \label{ap_a_1}
\end{align} 
where, in the inequality, we use \eqref{eq_union_prop} from Lemma \ref{lem_union_prop} and the fact that $\big\{ \mu_{t+1}(f_t(x_t,u_t,w_t))| w_t \in \mathcal{W}_t \big\} = \cup_{w_t \in \mathcal{W}_t} \big\{\mu_{t+1}(f_t(\bar{x}_t,u_t,w_t))\big\}$. Once again using \eqref{eq_union_prop} in the RHS of \eqref{ap_a_1}, we state that
\begin{align}
    &\mathcal{H}\big([[\hat{X}_{t+1}|x_t, u_t]],[[\hat{X}_{t+1}|\mu_t(x_t),u_t]]\big) \nonumber \\
    \leq &\max_{w_t \in \mathcal{W}_t, \bar{x}_t \in [[X_t|\mu_t(x_t)]]} d\Big(\mu_{t+1}\big(f_t(x_t,u_t,w_t)\big), \nonumber \\ &\mu_{t+1}\big(f_t(\bar{x}_t,u_t,w_t)\big)\Big) \nonumber \\
    \leq &\hspace{-5pt}\max_{w_t \in \mathcal{W}_t, \bar{x}_t \in [[X_t|\mu_t(x_t)]]} \Big(d\big(\mu_{t+1}(f_t(x_t,u_t,w_t)), f_t(x_t,u_t, \nonumber \\
    &w_t)\big) 
    + d\big(f_t(x_t, u_t, w_t), f_t(\bar{x}_t,u_t,w_t)\big) \nonumber \\
    &+ d\big(f_t(\bar{x}_t,u_t,w_t), \mu_{t+1}(f_t(\bar{x}_t,u_t,w_t))\big)\Big) \nonumber \\
    \leq &\gamma_{t+1} + 2L_{f_t} \cdot \gamma_t + \gamma_{t+1} =: \delta_t,
\end{align}
where, in the second inequality, we use the triangle inequality.
\end{proof}

\section*{Appendix B - Derivation of Approximate Information State for Partially Observed Systems}

In this appendix, we derive the values of $\epsilon_t$ and $\delta_t$ for all $t=0,\dots,T$, when an approximate information state is constructed using state quantization for a partially observed system. Recall from Section \ref{subsection:approx_examples} that given a set of quantized states $\hat{X}_t$ with parameter $\gamma_t$, quantization function $\mu_t: \mathcal{X}_t \to \hat{\mathcal{X}}_t$ and the conditional range $\Pi_t = [[X_t|M_t]]$ for all $t=0,\dots,T$, the approximate information state is given by $\hat{\Pi}_t = \nu_t(\Pi_t) = \big\{\mu_t(x_t) \in \hat{\mathcal{X}}_t|x_t \in \Pi_t\big\}$. Furthermore, it holds that $\max_{x_t \in \mathcal{X}_t} d(x_t, \mu_t(x_t)) \leq \gamma_t$ for all $t$. Note that the conditional range evolves as $\Pi_{t+1} = \bar{f}_t(\Pi_t, U_t, Y_{t+1})$ for all $t$ \cite{gagrani2017decentralized, gagrani2020worst}. Next, we state and prove the main result of this appendix.

\begin{theorem}
For all $t=0,\dots,T,$ let $\mu_t: \mathcal{X}_t \to \hat{\mathcal{X}}_t$ such that $\max_{x_t \in \mathcal{X}_t}d(x_t, \mu_t(x_t)) \leq \gamma_t$ and let $L_{c_T}$, $L_{\bar{f}_t}$, $L_{h_{t+1}}$, and $L_{f_t}$ be Lipschitz constants of the respective functions in the subscripts. Then, at time $t$, $\hat{\Pi}_t = \nu_t(\Pi_t)$ is an approximate information state with $\epsilon_T = 2L_{c_T} \cdot \gamma_T$ and $\delta_t = 2 \gamma_{t+1} + 2 L_{\bar{f}_t} \cdot L_{h_{t+1}} \cdot L_{f_t} \cdot \gamma_t$, where $\gamma_{T+1} = 0$.
\end{theorem}

\begin{proof}
For all $t=0,\dots,T$, let $m_t$, $P_t = [[X_t|m_t]]$, and $\hat{P}_t = \nu_t(P_t)$ be the realizations of the memory $M_t$, the conditional range $\Pi_t$ and the approximate information state $\hat{\Pi}_t$, respectively. Note that the conditional range $P_t$ satisfies \eqref{p1} and \eqref{p2} from Definition \ref{def_info_state}. To derive the value of $\epsilon_T$, we write the LHS of \eqref{ap1} using \eqref{p1} as
\begin{align}
    &\big|\max_{x_T \in [[X_T|m_T]]} c_T(x_T, u_T) - \max_{x_T \in [[X_T|\nu_T({P}_T)]]}c_T(x_T,u_T)\big| \nonumber \\
    =&\big|\max_{x_T \in P_T} c_T(x_T, u_T) - \max_{x_T \in [[X_T|\nu_t({P}_T)]])}c_T(x_T,u_T)\big| \nonumber \\
    \leq &L_{c_T} \cdot \mathcal{H}(P_T, [[X_T|\nu_t({P}_T)]]) \nonumber \\
    \leq &L_{c_T} \cdot \big(\mathcal{H}(P_T, \nu_T(P_T)\big) + \mathcal{H}\big(\nu_T(P_T), [[X_T|\nu_t({P}_T)]]) \big), \label{ap_b_1}
\end{align}
where, in the eqaulity, we use \eqref{p1}; in the first inequality, we use \eqref{eq_prelim} from Lemma \ref{lem_prelim}; and in the second inequality we use the triangle inequality for the Hausdorff metric. We can expand the first term in the RHS of \eqref{ap_b_1} as $\mathcal{H}(P_T, \nu_T(P_T)) = \mathcal{H}(P_T, \{\mu_T(x_T) \in \hat{\mathcal{X}}_T\;| \; x_T \in P_T\})$ $= \mathcal{H}\big(\cup_{x_T \in P_T} \{x_T\}, \cup_{x_T \in P_T}\{\mu_T(x_T) \in \hat{\mathcal{X}}_T\}\big) \leq \max_{x_T \in P_T}$ $d(x_T, \mu_T(x_T)) \leq \gamma_T$, where we use \eqref{eq_union_prop} from Lemma \ref{lem_union_prop} in the first inequality. We can also expand the second term in the RHS of \eqref{ap_b_1} as $\mathcal{H}(\nu_T(P_T), [[X_T|\nu_T({P}_T)]]) ) = \mathcal{H}\big(\nu_T(P_T), \{x_T \in \mathcal{X}_T| \min_{\bar{x}_T \in \nu_T(P_T)} d(x_T,\bar{x}_T) \leq \gamma_T \} \big) = \max_{x_T \in [[X_T|\nu_T(P_T)]]} \min_{\bar{x}_T \in \nu_T(P_T)} d(x_T, \bar{x}_T) \leq \gamma_T$, where the second equality holds by expanding the Hausdorff metric and noting that $\nu_T(P_T) \subseteq [[X_T|\nu_T(P_T)]]$. The proof is complete by substituting the results for both terms in the RHS of \eqref{ap_b_1}.

Next, to derive the value of $\delta_t$, we note that $P_t = \sigma_t(m_t)$. Then, we use the triangle inequality in the LHS of \eqref{ap2} to write that
\begin{align}
& \mathcal{H}\Big([[\nu_{t+1}(\Pi_{t+1})|m_t,u_t]],[[\nu_{t+1}(\Pi_{t+1})|\nu_t(\sigma_t(m_t)),u_t]]\Big) \nonumber \\
\leq & \mathcal{H}\Big([[\nu_{t+1}(\Pi_{t+1})|m_t,u_t]],[[\Pi_{t+1}|m_t,u_t]]\Big) + \nonumber \\
&\mathcal{H}\Big([[\Pi_{t+1}|m_t,u_t]],[[\Pi_{t+1}|\nu_t(\sigma_t(m_t)),u_t]] \Big) + \nonumber \\
& \mathcal{H}\Big([[\Pi_{t+1}|\nu_t(\sigma_t(m_t)),u_t]],[[\nu_{t+1}(\Pi_{t+1})|\nu_t(\sigma_t(m_t)),u_t]]\Big) \nonumber \\
\leq & 2\gamma_{t+1} + \mathcal{H}\Big([[\Pi_{t+1}|m_t,u_t]],[[\Pi_{t+1}|\nu_t(\sigma_t(m_t)),u_t]]\Big), \label{ap_b_2}
\end{align}
where, in the second inequality we use the face that $\mathcal{H}\big(P_{t+1}, \nu_{t+1}(P_{t+1})\big) \leq \gamma_{t+1}$, which was proved above. We can write the second term in the RHS of \eqref{ap_b_2} using \eqref{p2} from Definition \ref{def_info_state} as $\mathcal{H}\big([[\Pi_{t+1}|m_t,u_t]],[[\Pi_{t+1}|\nu_t(\sigma_t(m_t)),u_t]]\big) = \mathcal{H}\big([[\Pi_{t+1}|P_t,u_t]],[[\Pi_{t+1}|\nu_t(P_t),u_t]]\big)$. Furthermore, note that $[[\Pi_{t+1}|\nu_t(P_t), u_t]] = \big\{ \tilde{P}_{t+1} \in [[\Pi_{t+1}|\tilde{P}_t,u_t]]$ $| \tilde{P}_t \in [[\Pi_t|\nu_t(P_t)]] \big\} =  \cup_{\tilde{P}_t \in [[\Pi_t|\nu(P_t)]]} [[\Pi_{t+1}|\tilde{P}_t,u_t]]$. Next, we use \eqref{eq_union_prop} from Lemma \ref{lem_union_prop} to write that
\begin{align}
    & \mathcal{H}\Big([[\Pi_{t+1}|P_t,u_t]]),[[\Pi_{t+1}|\nu_t(P_t),u_t]]\Big) \nonumber \\
    \leq & \max_{\tilde{P}_t \in [[\Pi_t|\nu_t(P_t)]]}\mathcal{H}\Big([[\Pi_{t+1}|P_t,u_t]]),[[\Pi_{t+1}|\tilde{P}_t,u_t]]\Big) \nonumber \\
    \leq & L_{\bar{f}_t} \max_{\tilde{P}_t \in [[\Pi_t|\nu_t(P_t)]]}\mathcal{H}\Big([[Y_{t+1}|P_t,u_t]]),[[Y_{t+1}|\tilde{P}_t,u_t]]\Big) \nonumber \\
    \leq & L_{\bar{f}_t}  \hspace{-2pt} \cdot \hspace{-2pt} L_{h_{t+1}} \hspace{-3pt} \max_{\tilde{P}_t \in [[\Pi_t|\nu_t(P_t)]]} \hspace{-12pt} \mathcal{H}\Big([[X_{t+1}|P_t,u_t]]),[[X_{t+1}|\tilde{P}_t,u_t]]\Big), \label{ap_b_3}
\end{align}
where, the third inequality can be proven by substituting $Y_{t+1} = h_{t+1}(X_{t+1},V_{t+1})$ into the equation. We can further expand the third term in the RHS of \eqref{ap_b_3} and use \eqref{eq_union_prop} from Lemma \ref{lem_union_prop} to write that
\begin{align}
    & \max_{\tilde{P}_t \in [[\Pi_t|\nu_t(P_t)]]} \mathcal{H}\Big([[X_{t+1}|P_t,u_t]]),[[X_{t+1}|\tilde{P}_t,u_t]]\Big) \nonumber \\
    \leq &  \max_{\tilde{P}_t \in [[\Pi_t|\nu_t(P_t)]], w_t \in \mathcal{W}_t} \mathcal{H}\Big(  \{f_t(x_t,u_t,w_t)|x_t\in P_t\}, \nonumber \\
    & \{f_t(x_t,u_t,w_t)|x_t\in \tilde{P}_t\} \Big) \nonumber \\
    \leq &L_{f_t} \cdot  \max_{\tilde{P}_t \in [[\Pi_t|\nu_t(P_t)]]} \mathcal{H}\big(P_t, \tilde{P}_t\big) \nonumber \\
    \leq &L_{f_t} \cdot  \max_{\tilde{P}_t \in [[\Pi_t|\nu_t(P_t)]]} \big(\mathcal{H}(P_t, \nu_t(P_t)\big) + \mathcal{H}\big(\nu_t(P_t), \tilde{P}_t)\big) \nonumber \\
    \leq &2 L_{f_t} \cdot \gamma_t,
\end{align}
where, in the third inequality, we use the triangle inequality and in the fourth inequality we use the fact that for all $\nu_t(\tilde{P}_t) = \nu_t(P_t)$, for all $\tilde{P}_t \in [[\Pi_t|\nu_t(P_t)]]$. 
\end{proof}

\end{document}